\documentclass{amsart}
\usepackage{amsthm,amssymb,amscd,enumerate}

\usepackage{color}
\usepackage[normalem]{ulem}

\newtheorem{thm}{Theorem}[section]
\newtheorem{cor}[thm]{Corollary}
\newtheorem{lem}[thm]{Lemma}
\newtheorem{prop}[thm]{Proposition}
\theoremstyle{definition}

\newtheorem{example}{Example}
\theoremstyle{remark}
\newtheorem{rem}[thm]{Remark}
\numberwithin{equation}{section}

\begin{document}

\title[Index 3 Blocking Sets]{Blocking Sets of Index Three}
\author{William E. Cherowitzo}
\address{Department of Mathematical and Statistical Sciences \\
         University of Colorado Denver\\
         Campus Box 170\\
         P.O. Box 173364\\
         Denver, Colorado 80217-3364 \\
         USA}
\email{william.cherowitzo@ucdenver.edu}
\author{Leanne D. Holder}
\address{Department of Mathematics \\
         Rose-Hulman Institute of Technology \\
         5500 Wabash Avenue \\
         Terre Haute, Indiana 47803 \\
         USA}
\email{holder1@rose-hulman.edu}


\begin{abstract}
In this note we will provide proofs for the various statements that have been made in the literature about blocking sets of index three. Our aim is to clarify what is known about the characterization of these sets. Specifically, we provide constructions for all R\'edei blocking sets in PG($2,q$) of index three and explicitly determine all blocking sets of index three in PG($2,7$).
\end{abstract}
\maketitle
\section{Definitions}

A \textit{proper blocking set} in a plane is a set of points $S$ in that plane such that every line of the plane meets $S$ and $S$ contains no line. A blocking set is \textit{minimal} if it does not properly contain a smaller blocking set. The \textit{index} of a blocking set was introduced by Beutelspacher and Eugeni \cite{BeEu:85} as the minimum number of lines that can cover the blocking set. It is easy to see that the index of a proper blocking set is at least 3. In this case, there are only two possibilities: the three lines can form a triangle or they can be concurrent.

Following Cameron \cite{PJC:85} we define, for an abelian group $G$ of order $n$ and a positive integer $m$, the relation $G \rightarrow m$ if there are nonempty subsets $A$, $B$ and $C$ of $G$ such that
\begin{enumerate}[(i)]
\item $0 \notin A + B + C$ (or $1 \notin ABC$ if $G$ is written multiplicatively);
\item $(A,B,C)$ is maximal subject to (i); that is, no element can be adjoined to any of the three sets without violating (i);
\item $|A| + |B| + |C| = m$.
\end{enumerate}

\begin{example}[Sz\"onyi \cite{TS:91}] For odd primes $p$, let  $G = (\mathbb{Z}_p,+) \times (\mathbb{Z}_p,+)$, and let $A = \{(x,x^2)\colon x \in \mathbb{Z}_p \}$, $B = \{(-x,-x^2)\colon x \in \mathbb{Z}_p \},\mbox{ and } C = \{(0,y)\colon y \in \mathbb{Z}_p^*\}$. We can see that $A + B = G \setminus C$, $A + C = G \setminus A = G \setminus \{-B\}$ and $B + C = G \setminus B = G \setminus \{-A\}$. This proves that $G \rightarrow 3p - 1$.
\end{example}

\begin{prop}[Cameron \cite{PJC:85}] \label{Pr:arrow} Let $G$ be an abelian group of order $n$. Then
\begin{enumerate}[(a)]
\item if $G \rightarrow m$, $3(\sqrt{n + \frac{1}{4}}) \le m \le \frac{3n}{2}$.
\item if $\theta: G \rightarrow H$ is an epimorphism and $H \rightarrow m$ then,
\begin{equation*}
              G \rightarrow \frac{|G|m}{|H|}.
\end{equation*}
\item $|G| \rightarrow n+d$ for any proper divisor $d$ of $n$.
\end{enumerate}
\end{prop}

\begin{proof}
$ $
\begin{enumerate}[(a)]
\item We will write $G$ additively. If $a \in A$ then $-a \notin B + C$, so for each $b \in B, -a - b \notin C$. Thus, $|G| - |C| \ge |B|$, that is $ n \ge |B| + |C|$. Similarly we have, $n \ge |A| + |B|$ and $n \ge |A|+|C|$. Adding these three inequalities gives $3n \ge 2m$ and we obtain the upper bound. Now, let $|A| = x, |B| = y$ and $|C| = z$. Since $-C \subseteq G \setminus (A+B)$, we have  $|A+B \cup - C| = |A+B| + |-C| \le xy + z $. Suppose $t \in G$ and $t \notin A+B \cup - C$. Since $t \notin -C$ we have $-t \notin C$. Let $C^* = C \cup \{-t\}$ and note that $0 \notin A + B + C^*$ contradicting the maximality of $(A,B,C)$. Thus, $n = |A+B \cup - C| \le xy + z$. Adding the three similar inequalities gives $3n \le xy + xz + yz + x + y + z$. It is easy to see that the maximum value of the right hand side occurs when $x = y = z = \frac{m}{3}$ from which we obtain $3n \le \frac{1}{3}m^2 + m$ and the lower bound follows.
\item Let $A', B'$ and $C'$ realize $H \rightarrow m$, and let $A = \theta^{-1}(A'), B = \theta^{-1}(B')$ and $C = \theta^{-1}(C')$. If $0 = a + b + c$ with $a \in A, b \in B \mbox{ and } c \in C$ then $0 = \theta (0) = \theta (a+ b + c) = \theta (a) + \theta (b) + \theta(c) \in A' + B' + C'$. Maximality follows from similar considerations. The last statement is a consequence of the fact that the size of the kernel of $\theta = \frac{|G|}{|H|}$.
\item Notice that for any group $J$ of order $j$ we have $J \rightarrow j+1$ since (writing the group additively) $A = B = \{0\}, \mbox{ and } C = J \setminus \{0\}$ satisfies the definition. For any proper divisor $d$ of $n = |G|$, let $H$ be a subgroup of order $\frac{n}{d}$. We have that $H \rightarrow \frac{n}{d} + 1$ and since $G$ is abelian we can let $\theta : G \rightarrow H$ be the canonical epimorphism. The result now follows from (b).
\end{enumerate}
\end{proof}

Sz\"onyi \cite{TS:91} was able to provide a partial converse to the third part of the above proposition by using a theorem of Kneser. This result is

\begin{thm}[Kneser] Let $A, B$ be two complexes (i.e. subsets) of the abelian group $G$. Then there is a subgroup $H$ of $G$ such that
\begin{enumerate}[i)]
\item $A + B = A + B + H$;
\item $|A + B| \ge |A + H| + |B + H| - |H|$.
\end{enumerate}
\end{thm}

We then have,

\begin{prop}[Sz\"onyi \cite{TS:91}] If $G \rightarrow m$ with $m > |G| + 1$, then there is a subgroup $H$ of $G$ such that $m = |G| + |H|$.
\end{prop}

\begin{proof} Let $(A,B,C)$ realize $G \rightarrow m$ and choose a subgroup $H$ which satisfies Kneser's theorem with respect to $A,B$. As $A+B \neq G$ we have $H \neq G$. Since $A+B \cap -C = \emptyset$, if $H = 0$ we would have that $n \ge |A+B| + |-C| \ge |A| + |B| - 1 + |C| = m - 1$, a contradiction. By the maximality of $(A,B,C)$ we have that $A = A + H, B = B + H, C = C + H \mbox{ and } -C = G \setminus (A+B)$. Consider $A + B + C + H = (A + B + H) + C = A + B + C$, but by commutivity we have $A+B+(C+H) = A + (B+H) +C = (A+H) + B + C = A + B + C$. Since $H$ is a subgroup, $X \subseteq X + H$ for any subset $X$. If $ X \neq X + H$ then an element could be added to $X = A, B \mbox{ or } C$ violating the maximality of these sets.  Also $-C = -(C+H) = -C - H = -C + H$. We also note that if $G \neq (A + B) \cup (-C)$ we would have $n > |A+B| + |-C| \ge |A+B| + |-C| + |H| \ge |A| + |B| - |H| + |C| + |H| = m$ a contradiction (this follows since all the sets involved are unions of cosets of $H$ as is $G$). This same contradiction shows that $|A+B| = |A| + |B| - |H|$ and the conclusion follows.
\end{proof}

It follows from the details of this proof that,

\begin{cor} \label{Cor:Structure} If $G \rightarrow m$ with $m > |G| + 1$, then there exists a subgroup $H$ such that if $(A,B,C)$ realize $G \rightarrow m$, $A, B \mbox{ and } C$ are unions of cosets of $H$. \qed
\end{cor}

\section{Blocking Sets of Index 3}

The interest in the concepts of the last section is due to the following theorem.

\begin{thm}[Cameron \cite{PJC:85}] \label{Th:index3} Let $S$ be a minimal blocking set in $PG(2,q)$ of index 3. Then one of the following holds:
\begin{enumerate}[(i)]
\item $|S| = 2q$;
\item $|S| = 3(q-1)$;
\item $|S| = 3q + 1 - m$; where $(GF(q),+) \rightarrow m$ and $q > 2$; (corrected from original)
\item $|S| = 3q - m$, where $GF(q)^{\times} \rightarrow m$.
\end{enumerate}
\end{thm}

\begin{proof}
Clearly, the three lines which contain the blocking set $S$ either form a triangle or are concurrent.

 In the concurrent case, the point of concurrency must be in $S$ as otherwise we would have $q = 2$ and a projective plane of order 2 contains no non-trivial blocking sets. Suppose that two of the lines, say $\ell$ and $m$, each contain $q$ points of $S$ (the maximum possible) with $P = \ell \cap m \in S$ and $R$ and $Q$ the points not in $S$ on $\ell$ and $m$ respectively. The only line not blocked by an $S$ point of $\ell \cup m $ is $RQ$. There must exist one, and by minimality only one, point of $S$ other than $R$ or $Q$ on the line $RQ$. The line joining this point and $P$ is a third line which with $\ell$ and $m$ covers $S$. For this blocking set of index 3 we have $|S| = 2q$. To deal with the general case of concurrent lines, we can map $P \rightarrow (0,1,0)$, $\ell \rightarrow X = 0$, $m \rightarrow X = 1$ and the third line to the line at infinity, $ Z = 0$. Let the points of $S$ be described by $\{(0,a,1) \colon a \in A' \} \cup \{(1,-b,1) \colon b \in B'\} \cup \{(1,c,0) \colon c \in C'\} \cup \{(0,1,0)\}$. By the previous special case we can assume that at most one of the lines contains $q$ points of $S$, so we may assume that none of the sets $A',B'$ or $C'$ contains all $q$ elements of $GF(q)$, and thus their complements, $A, B$ and $C$ respectively, are non-empty. $S$ will be a blocking set provided that $-(a+b) \in C'$ whenever $a \notin A'$ and $b \notin B'$. That is, $0 \notin A+B+C$. With $|A| + |B| + |C| = m$, $|S| = q - |A| + q - |B| + q - |C| + 1 = 3q - m + 1 $.

 We now turn to the case that the three lines form a triangle. We will first consider the various possibilities when some of the vertices of the triangle are not in $S$. Assume that there is exactly one vertex, say $P$, which is not in $S$. The $q-1$ lines through $P$ other than the triangle sides must be blocked by distinct points on the third side of the triangle. The remaining two points on this third side are in $S$ by assumption, so this entire line is contained in $S$, a contradiction. If exactly two vertices are not in $S$ then the triangle sides opposite these points must contain $q$ points of $S$, and we are in the special case of the concurrent line situation. If the three vertices of the triangle are not in $S$ then the side of the triangle opposite any vertex must contain $q-1$ points of $S$. For this configuration, we have $|S| = 3(q-1)$. We may now assume that that the three vertices of the triangle are all in $S$. We can map these vertices to the points with coordinates $(1,0,0), (0,1,0)$ and $(0,0,1)$ and the remainder of the points of $S$ define the subsets $A$, $B$, and $C$ of $GF(q)^{\times}$ by $\{(-1,x,0) \colon x \notin A \} \cup \{ (0,-1,y) \colon y \notin B\} \cup \{(z,0,-1) \colon z\notin C\}$. As each side of the triangle must contain at least one point not in $S$, the sets $A,B$ and $C$ are all non-empty. The points $(-1,x,0), (0,-1,y)$ and $(z,0,-1)$ with $x \in A, y \in B$ and $z \in C$ are collinear (and hence on a line not blocked by the points of $S$) if and only if, $xyz = 1$. So, with $|A| + |B|+ |C| = m$, $S$ is a blocking set if and only if $(A, B,C)$ realizes $GF(q)^{\times} \rightarrow m$. In this case we have that $|S| = q-1 - |A| + q - 1 - |B| + q - 1 - |C| + 3 = 3q - m$.

\end{proof}

\section{R\'edei Blocking Sets of Index 3}

Our interest is in minimal \textit{R\'edei blocking sets} of index 3, that is, those minimal blocking sets in $PG(2,q)$ with $q+n$ points, achieving a maximum of $n$ points on a line called a \textit{R\'edei line}. A blocking set is \textit{minimal} if each of its points is \textit{essential}, that is, lies on a tangent line (1-secant) of the set. For these blocking sets we have $n \le q$, so the size of the blocking set is at most $2q$. In light of Corollary \ref{Cor:Structure} we are interested in the following cases of Theorem \ref{Th:index3}: (i), (ii) when $q = 3$, (iii) when $m = q+1$ and (iv) when $m = q$ or $m = q+1$. In all of these cases the R\'edei line contains either $n = q$ or $n = q-1$ points. We shall classify the R\'edei blocking sets of index 3 with $n = q$ or $q-1$ in this section and amplify the description of those falling under Corollary \ref{Cor:Structure} in the next section.

Index 3 blocking sets are by definition contained in three lines. These three lines can either be concurrent or form a triangle. We shall use the term \textit{triad} to refer to the concurrent line case. After introducing coordinates and using the fundamental theorem of projective geometry, we may assume without loss of generality that in the triad case the point of concurrency, necessarily in the blocking set since $q > 2$, has coordinates $(0,1,0)$, the R\'edei line has equation $z = 0$, the remaining two lines have affine equations $x = 0$ and $x = 1$, and the point $(1,0,0)$ is not in the blocking set. For the triangle case we can take as our standard configuration, the R\'edei line as $z=0$, the other two lines with affine coordinates $x=0$ and $y=0$, and the point $(1,1,0)$ not in the blocking set. Note that in the triangle case, the points $(0,1,0), (1,0,0)$ and $(0,0,1)$ are all in the blocking set.

Since a blocking set contains no line, the R\'edei line must contain at least one point which is not in the blocking set. Each of the other $q$ lines through a point of the R\'edei line not in the blocking set must be blocked by points on the two other lines of the configuration. Let $\mathcal{A}$ denote the affine points of the blocking set on the line $x = 0$ in either case, that is $\mathcal{A} = \{(0,a)\colon a \in A\}$ where $A \subseteq GF(q)$. The remaining points of the blocking set will be denoted by $\mathcal{B}$ and is either a subset of points on the line $x = 1$ in the triad case, or a subset of points on $y = 0$ in the triangle case.

We first consider the triad case where, $\mathcal{B} = \{(1,b)\colon b \in B\}$ with $B = GF(q)\setminus A$. The triad configuration is stabilized by any $((\infty), y=k)$-homology, so we may assume that $0 \in A$ and if $|A| \ge 2$ that $1 \in A$ as well (see Lemma \ref{Lm:triadinterchange} below). If the R\'edei line contains $q$ points of the minimal blocking set, each of these points must be on a line determined by a point of $\mathcal{A}$ and a point of $\mathcal{B}$ since otherwise the point would not be an essential point of the blocking set. More specifically, define $f: A \times B \mapsto GF(q)^*$ by $f(a,b) = b - a$. Points $(0,a)$ of $\mathcal{A}$ and $(1,b)$ of $\mathcal{B}$ determine a line meeting $\ell_{\infty}$ (the R\'edei line) at the point $(1,f(a,b),0)$. Since $0 \in A$ we have $B \subseteq Im(f)$. As $A$ and $B$ are disjoint, $0 \not\in Im(f)$.

\begin{lem} \label{Lm:Triad}
If $a_0 \in A \setminus\{0\}$ then $a_0 \in Im(f)$  if and only if $\exists a_0^* \in A \setminus\{0\}$ such that $a_0 + a_0^* \not \in A$.
\end{lem}

\begin{proof} We have that $a_0 \not\in Im(f)$ if and only if $\nexists (a,b) \in A \times B$ such that $a_0 = b - a$. Thus, $b \neq a_0 + a$ for any $a \in A$, so $a_0 + a \in A, \forall a \in A$.
\end{proof}

\begin{lem} \label{Lm:triadinterchange}
In the triad case, the group stabilizing the configuration acts transitively on the set of ordered pairs of affine points, both of which lie on $x = 0$ or $x = 1$.
\end{lem}

\begin{proof}
The collineation $g$ given by
\begin{equation}
           (x,y,z) \mapsto (x,y,z)
           \left (
           \begin{matrix}
           -1 & 0 & 0 \\
           0 & 1 & 0 \\
           1 & 0 & 1
           \end{matrix}
           \right )
\end{equation}
fixes $(0,1,0)$ and $(1,0,0)$, stabilizes $z = 0$ and interchanges $x = 0$ and $x = 1$ in either characteristic.

Let $((x,a),(x,b))$ and $((x,c),(x,d))$ be two pairs of points on the same line (with $x = 0$ or $1$). Let $H$ and $K$ be homology groups with center $(0,1,0)$ and axes $y = c$ and $y = n$ respectively. These homology groups stabilize the triad configuration and stabilize the lines $x = 0$ and $x = 1$ individually, while acting transitively on the affine points of these lines. Apply an homology $k \in K$ so that $(x,a) \mapsto (x,c)$. Using an homology $h\in H$, map $(x,b)^k \mapsto (x,d)$, noting that $(x,c)$ is fixed under $h$. The composition of these two homologies maps the first pair of points onto the second pair.

If the two pairs of points lie on different lines, first apply collineation $g$ and then use the homologies as above.
\end{proof}

\begin{prop} \label{Pr:Triad}
In the triad case, with the notation of this section, the R\'edei line contains $q$ points of the blocking set if and only if for each non-zero $a_0 \in A$, there exists $a_0^* \in A \setminus\{0\}$ such that $a_0 + a_0^* \not \in A$. Furthermore, the R\'edei line contains $q-1$ points of the blocking set if and only if $q$ is even and $A$ is a proper union of additive cosets of the subgroup $\langle 0,1\rangle$.
\end{prop}

\begin{proof}
By Lemma \ref{Lm:Triad} the condition implies that all non-zero elements of $A$ are in $Im(f)$, thus we have that $Im(f)$ contains all non-zero elements of $GF(q)$. The point $(0,1,0)$ is also an essential point of the blocking set since $q > 2$, and this gives $q$ essential points of the blocking set on the R\'edei line.

If there is a non-essential point $P$ of the blocking set on the R\'edei line, then $P$  could be on no 3-secant line of the blocking set, since the number of 3-secant lines through $P$ equals the number of tangent lines through $P$. Thus, there would exist a nonzero $a_1 \in A$ which is not in $Im(f)$. By Lemma \ref{Lm:Triad}, we would have that $a_1 + a \in A, \forall a \in A$. In particular, each element of the additive subgroup $\langle a_1\rangle$ would be in $A$ and would have the same property as $a_1$. By the same Lemma, none of these elements would be in $Im(f)$. Under the assumption that the R\'edei line contains $q-1$ essential blocking set points, the size of this subgroup must be two, i.e. $2a_1 = 0$, and so the field must have characteristic 2. Since $A$ is closed under the addition of $a_1$, it must be a union of cosets of $\langle 0,a_1 \rangle$. $A$ may not be empty, nor equal to $GF(q)$.  By Lemma \ref{Lm:triadinterchange} we may take $a_1 = 1$ without any loss of generality.
\end{proof}

We now turn to the triangle case. We are assuming that the point $(1,1,0)$ on the R\'edei line is not in the blocking set. The origin, $(0,0,1)$, is an essential point of the blocking set and we will take the set $A$ as a subset of $GF(q)^* = GF(q)\setminus\{0\}$. To block the remaining $q-1$ lines through $(1,1,0)$ we require that $\mathcal{B} = \{(b,0)\colon b \in B\}$ where $B = GF(q)^* \setminus \{-A\}$. As neither $A$ or $B$ can be empty, the points $(0,1,0)$ and $(1,0,0)$ are essential points of the blocking set on the R\'edei line. Other points of the blocking set on the R\'edei line are essential provided they lie on a 3-secant blocking set since the number of tangent lines through the point equals the number of 3-secant lines through that point. Define $h: A \times B \mapsto GF(q)^*$ by $h(a,b) = - \frac{a}{b}$. Points $(0,a)$ of $\mathcal{A}$ and $(b,0)$ of $\mathcal{B}$ determine a line meeting $\ell_{\infty}$ (the R\'edei line) at the point $(1,h(a,b),0)$. As $-A$ and $B$ are disjoint, $1 \not\in Im(h)$. By using a $((0,0), \ell_{\infty})$-homology, we can assume that $-1 \not\in A$ and so $1 \in B$. Thus, we have $-A \subseteq Im(h)$.

\begin{lem} \label{Lm:Triangle}
If $b_0 \in B \setminus\{1\}$ then $b_0 \in Im(h)$  if and only if there exists $b_0^* \in B \setminus\{1\}$ such that $b_0 b_0^* \not \in B$.
\end{lem}

\begin{proof} We have $b_0 \not\in Im(h)$ if and only if there does not exist an $(a,b) \in A \times B$ such that $b_0 = - \frac{a}{b}$. Thus, $bb_0 \neq - a$ for any $b \in B$, so $bb_0 \in B, \forall b \in B$.
\end{proof}

\begin{prop} \label{Pr:Triangle}
In the triangle case, with the notation of this section, the R\'edei line contains $q$ points of the blocking set if and only if for each non-identity $b_0 \in B$, there exists $b_0^* \in B \setminus\{1\}$ such that $b_0b_0^*  \not \in B$. Furthermore, the R\'edei line contains $q-1$ points of the blocking set if and only if $q$ is odd and $B$ is a proper union of multiplicative cosets of the subgroup $\langle 1,-1\rangle$ of $GF(q)^*$.
\end{prop}

\begin{proof}
By Lemma \ref{Lm:Triangle} the condition implies that all elements of $B$ other than $1$ are in $Im(h)$, thus we have that $Im(h)$ contains all non-identity elements of $GF(q)^*$. As the points $(0,1,0)$ and $(1,0,0)$ are also essential points of the blocking set, this gives $q$ essential points of the blocking set on the R\'edei line.

If there is a non-essential point of the blocking set on the R\'edei line, there would exist a  $b_1 \in B \setminus \{1\}$ which is not in $Im(h)$. By Lemma \ref{Lm:Triangle}, we would have that $bb_1\in B, \forall b \in B$. In particular, each element of the multiplicative subgroup $\langle b_1\rangle$ would be in $B$ and would have the same property as $b_1$. By the same Lemma, none of these elements would be in $Im(h)$. Under the assumption that the R\'edei line contains $q-1$ essential blocking set points, the size of this subgroup must be two, i.e. $b_1^2 = 1$, and so $b_1 = -1$ and the field must have odd characteristic. Since $B$ is closed under the multiplication by $b_1$, it must be a union of the multiplicative cosets of $\langle 1,-1 \rangle$. $B$ may not be empty, nor equal to $GF(q)^*$.
\end{proof}

\begin{rem} \label{Rm:Triangle}
Recall that the collineation,
\begin{equation}
           (x,y,z) \mapsto (x,y,z)
           \left (
           \begin{matrix}
           0 & 1 & 0 \\
           1 & 0 & 0 \\
           0 & 0 & 1
           \end{matrix}
           \right )
\end{equation}
fixes $(0,0,1)$ and $(1,1,0)$, stabilizes $z = 0$ and interchanges $x = 0$ and $y=0$ in either characteristic.
\end{rem}


\section{Constructions}

As reported in Sz\"onyi, G\'acs and Weiner \cite{SzGaWe:03} we have a construction credited to Megyesi\footnote{The usual reference given for this is R\'edei \cite{LR:73}, but R\'edei states that Megyesi only gave an alternate, more complicated form and the additive subgroup case does not appear at all in the monograph. What does appear, in an example of the use of lacunary polynomials, is clearly isomorphic, but not identical, to the construction stated here.}, namely:
\begin{thm} \label{Th:Meg} Let $d$ be a divisor of $q-1$ and let $H$ be a multiplicative subgroup of size $d$. Consider the set $U = \{(0,h) \colon h \notin H \} \cup \{(g,0)\colon g \in H\}$. Then $U$ determines exactly $q+1-d$ directions. Similarly, if $d$ divides $q$, then using additive subgroups and two parallel lines (instead of the two axes) one can construct R\'edei type blocking sets of size $2q+1-d$ (again $q+1-d$ directions).
\end{thm}

This is a special case of a slightly more general construction:
\begin{thm} \label{Th:const} Let $d$ be a divisor of $q-1$ and let $H$ be a multiplicative subgroup of $GF(q)$ of size $d$. Let $A = g_0 H$ and $B_1 = g_1 H$ be arbitrary cosets of $H$ in $GF(q)^*$. Consider the set $U = \{(0,h,1) \colon h \notin A \} \cup \{(-g,0,1)\colon g \in B_1\}$. Then $U$ determines exactly $q+1-d$ directions. In particular these directions correspond to all the points on the line at infinity ($Z = 0$) except those in  $W = \{(1,c,0) \colon c \in C \}$ where $C = \frac{g_0}{g_1} H$.  Denote this set of points by $W_1$. The R\'edei type blocking set, $U \cup W_1$ has size $2q+1-d$. Similarly, if $d$ divides $q$, then using cosets of an additive subgroup of order $d$ and two parallel lines (instead of the two axes) one can again construct R\'edei type blocking sets of size $2q+1-d$.
\end{thm}
\begin{proof}
Let $B = GF(q)^* \setminus B_1$. Then $(A, B, C)$ realizes $GF(q)^* \rightarrow m = q - 1 + |H|$, and so, by Theorem \ref{Th:index3} (iv), gives rise to a R\'edei type blocking set of size $3q-m = 2q + 1 -d$.

We turn now to the details of the additive case. Let $d$ be a divisor of $q$ and let $H$ be an additive subgroup of $GF(q)$ of size $d$. Let $A = g_0 + H$ and $B_1 = g_1 + H$ be arbitrary cosets of $H$ in $(GF(q),+)$. Let $U = \{(0,a,1) \colon a \notin A \} \cup \{(1,b,1) \colon b \in B_1\}$. Then $U$ determines exactly $q+1-d$ directions. In particular, these directions correspond to all the points on the line at infinity ($Z = 0$) except those in the set $W = \{(1,c,0) \colon c \in C \}$ where $C = (g_0 - g_1) + H$. Let $B = GF(q) \setminus B_1$. Then $(A,B,C)$ realizes $(GF(q),+) \rightarrow m = q + |H|$, and so, by Theorem \ref{Th:index3} (iii), gives rise to a R\'edei type blocking set, $U \cup W$, of size $ 3q + 1 - m = 2q+1-d$.
\end{proof}

Rephrasing this construction in a more usable form gives:

\begin{cor} \label{Co:any coset}
Let $d$ be a proper divisor of $q$ (resp. $q-1$) and $H$ a subgroup of $G =(GF(q),+)$ (resp. $G = (GF(q)^*,\cdot)$) of size $d$. Let $C$ be any coset of $H$, and $B$ be the complement in $G$ of any coset $D$ of $H$. Then there exists a coset $A$ of $H$ such that the union of the sets $\mathcal{A} = \{(0,a,1)\colon a \in A\}$, $\mathcal{B} = \{(1,b,1)\colon b \not\in D\}$ (resp., $\mathcal{B} = \{(b,0,1)\colon b \not\in D\}$) and $\mathcal{C} = \ell_{\infty} \setminus \{(1,c,0)\colon c \in C\}$ is a R\'edei type blocking set of size $2q+1-d$. \qed
\end{cor}

If in Theorem \ref{Th:Meg} $d = \frac{q-1}{2}$ then the constructed R\'edei blocking set is called a \textit{projective triangle}, while if $d = \frac{q}{2}$ it is called a \textit{projective triad}.

Propositions \ref{Pr:Triad} and \ref{Pr:Triangle} characterize the R\'edei blocking sets having $q$ points on the R\'edei line. We now present some constructions in this situation.

\begin{example} Let $q$ be an odd prime power greater than $3$ and fix $t < \frac{q-1}{2}$. With $\alpha$ a primitive element of $GF(q)$, define $A = \{{\alpha}^i\colon 1 \le i \le t\}$ and $B = GF(q)^* \setminus \{-A\}$.
\end{example}

\section{PG(2,7)}
 While there is considerable literature on the spectrum of sizes of blocking sets, only in the smallest planes have complete determinations of all blocking sets been made. Often one is satisfied if a construction is found for each size, but there is some work on proving non-existence or uniqueness for a given size blocking set. We shall investigate the blocking sets in $PG(2,7)$, with an eye towards classifying all the R\'edei blocking sets of index 3 in this plane. Much of the previous work in this plane has been carried out by Innamorati and Maturo \cite{InMa:91} and \cite{InMa:91a}.

 The theoretical limits for the size of a blocking set in $PG(2,7)$ range from 12 to 19. There are examples for each size (see \cite{SzGaWe:03}).  There are R\'edei blocking sets of sizes 12-14. There are two blocking sets of size 12, the projective triangle and a non-R\'edei blocking set. The four possible sets of this size constructed from Theorem \ref{Th:const} with $d = 3$ are all isomorphic (giving the projective triangle). In \cite{PJC:85} Cameron states that the constructions of Theorem \ref{Th:index3} provide 11 inequivalent blocking sets of size 14 in $PG(2,7)$. The blocking set of size 19 is unique.

 We shall show the uniqueness of the R\'edei blocking set of index 3 and size 13 and enumerate those of size 14.

 For a size 13 R\'edei blocking set, the R\'edei line contains $q-1$ of its points. By Proposition \ref{Pr:Triangle}, the set $B$ consists of a union of multiplicative cosets of $\langle 1, -1 \rangle$, that is, a proper subset of $\left\{\{1,6\},\{2,5\},\{3,4\}\right\}$. If $B$ contains only one coset then we may take it to be the subgroup $\langle 1,6\rangle$. If $B$ contains two cosets then by applying the collineation of Remark \ref{Rm:Triangle}, we interchange the roles of the sets $A$ and $B$. The slopes of the lines determined by the new sets $\mathcal{A}$ and $\mathcal{B}$ are the reciprocals of those determined by the old sets, so the set of points of the R\'edei blocking set on the R\'edei line remains the same. The new set $B$ will consist of just one coset and so, projectively equivalent to $\langle 1,6 \rangle$.

 The size 14 R\'edei blocking sets of index 3 are described by Propositions \ref{Pr:Triad} and \ref{Pr:Triangle}. We shall examine these sets in terms of $| A |$ in the triad case and $| B |$ in the triangle case.  For $| A | = 1$, we may assume that $A = \{0\}$, and for $| B | = 1$ we can take $B = \{1\}$ (the conditions are satisfied vacuously). These two cases are projectively equivalent. In the triad case, when $| A | = 2$ we may take $A = \{0,1\}$, by Lemma \ref{Lm:Triad}, and in the triangle case we may take $B = \{1,b\}$ with $b \neq -1$. There are two inequivalent blocking sets in this latter situation given by $b \in \{2,4\}$ or $b \in \{3,5\}$. When $| A | = 3$ there are two inequivalent blocking sets given by $A = \{0,1,5\}$ or $A = \{0,1,6\}$. With $| B | = 3$ there are two inequivalent blocking sets given by $B = \{1,2,3\}$ or $B = \{1,2,5\}$. Larger values of $| A |$ and $| B |$ give blocking sets that are projectively equivalent to those we have seen by Lemma \ref{Lm:triadinterchange} and Remark \ref{Rm:Triangle}. There are thus eight projectively inequivalent R\'edei blocking sets of index 3 of size 14 in $PG(2,7)$. This example implies that ``projective equivalence'' may be too fine an equivalence relation to be used in distinguishing these blocking sets.
 
 There are three size 14 non-equivalent non-R\'edei blocking sets of index 3 in $PG(2,7)$. In the triangle case we have $A = \{1,3\}, B =\{1,3\}$ and $C = \{2,3,6\}$, while in the triad case we have either $A = \{0,1\}, B = \{0,1\}$ and $C = \{1,2,3,4\}$ or $A = \{0,1\}, B = \{0,1,2\}$ and $C = \{1,2,3\}$. This accounts for all 11 examples reported by Cameron \cite{PJC:85}.

\end{document}